\documentclass[12pt,fleqn]{article}
\usepackage{amsmath,amssymb,amsfonts,amsthm}
\usepackage[english]{babel}
\usepackage{hyperref}
\author{Xavier Lamy}
\title{Uniaxial symmetry in nematic liquid crystals}
\usepackage{fouriernc}

\usepackage[margin=35mm]{geometry}
\newtheorem{thm}{Theorem}[section]
\newtheorem{prop}[thm]{Proposition}
\newtheorem{lem}[thm]{Lemma}

\theoremstyle{definition}
\newtheorem{rem}[thm]{Remark}

\def\R{\mathbb{R}}

\begin{document}
\maketitle

\begin{abstract}
Within the Landau-de Gennes theory of liquid crystals, we study theoretically the equilibrium configurations with uniaxial symmetry. We show that the uniaxial symmetry constraint is very restrictive and can in general not be satisfied, except in very symmetric situations. For one- and two-dimensional configurations, we characterize completely the uniaxial equilibria: they must have constant director. In the three dimensional case we focus on the model problem of a spherical droplet with radial anchoring, and show that any uniaxial equilibrium must be spherically symmetric.
It was known before that uniaxiality can \textit{sometimes} be broken by energy minimizers. Our results shed a new light on this phenomenon: we prove here that in one or two dimensions uniaxial symmetry is \textit{always} broken, unless the director is constant. Moreover, our results concern all equilibrium configurations, and not merely energy minimizers.
\end{abstract}

\section{Introduction}

Nematic liquid crystals are composed of rigid rod-like molecules which tend to align in a common preferred direction. For a macroscopic description of such orientational ordering, several continuum theories are available, relying on different order parameters.

The state of alignment can be simply characterized by a director field $n$ with values in the unit sphere $\mathbb S^2$, corresponding to the local preferred direction of orientation. Within such a description, topological constraints may force the appearance of defects: regions where the director field is not continuous. 
To obtain a finer understanding of such regions, one needs to introduce a scalar order parameter $s$, corresponding to the degree of alignment along the director $n$. 
%This description uses three degrees of freedom. 
However, the $(s,n)$ description only accounts for uniaxial nematics, which correspond to a symmetric case of the more general biaxial nematic phase. To describe biaxial regions, a tensorial order parameter $Q$ is needed. 
%It possesses five degrees of freedom: three degrees of freedom for the two principal directions, and two further degrees of freedom to describe the degree of alignment in these directions. 
Biaxiality has been used to theoretically describe defect cores 
\cite{mkaddemgartland00, kraljvirgazumer99, penzenstadlertrebin89, sonnetkilianhess95,lucarey07a,lucarey07b} and material frustration \cite{palffymuhoraygartlandkelly94, bisigartlandrossovirga03,ambrozicbisivirga08}, and has been observed experimentally \cite{madsendingemansnakatasamulski04, acharyaprimakkumar04}.

The $n$ and $(s,n)$ descriptions can both be interpreted within the $Q$-tensor description. The tensorial order parameter $Q$ describes different degrees of symmetry: isotropic, uniaxial or biaxial. The isotropic case $Q=0$ corresponds to the full symmetry group $G=SO(3)$. The uniaxial case corresponds to a broken symmetry group $H\approx O(2)$. And the biaxial case corresponds to a further broken symmetry group with 4 elements. The $(s,n)$ description amounts to restricting the order parameter space to uniaxial or isotropic $Q$-tensor: only  $Q$-tensors which are `at least $O(2)$-symmetric' are considered. The $n$ description arises in the London limit, since the space of degeneracy is $G/H\approx\mathbb S^2/\lbrace \pm 1\rbrace$ (see for instance \cite[Section~2]{trebin82} for more details).

 In  physical systems presenting some symmetry, existence of symmetric equilibrium configurations is a common phenomenon: such configurations can be obtained by looking for a solution with a special symmetrical \textit{ansatz}. In some cases this phenomenon can be formalized mathematically as a Principle of Symmetric Criticality \cite{palais79}.  
In the present paper we investigate whether the same principle applies to uniaxial symmetry in nematic liquid crystals: do there exist uniaxial $Q$-tensor equilibrium configurations? or is the uniaxial symmetry always broken?

We consider a Landau-de Gennes free energy. We do not work with the usual four-terms expansion of the bulk free energy but with a general frame invariant bulk free energy.

We start by considering the case of one- or two-dimensional configurations: that is, configurations exhibiting translational invariance in at least one direction of space \cite{palffymuhoraygartlandkelly94, bisigartlandrossovirga03,ambrozicbisivirga08,kraljvirgazumer99, sonnetkilianhess95}. In Theorem~\ref{thm2D} we describe completely the one- or two-dimensional uniaxial equilibrium configurations: these are essentially only the configurations with constant director field $n$. In particular, \textit{even if the boundary conditions enhance uniaxial symmetry, the uniaxial order is destroyed in the whole system}, unless the director field is uniform.

The three dimensional case is more complex. While in one and two dimensions the uniaxial configurations are essentially trivial, there does exist a non trivial uniaxial configuration in three dimensions: namely, the so-called radial hedgehog \cite{mkaddemgartland00,schopohlsluckin88}, which corresponds to a spherically symmetric configuration in a sperical droplet of nematic, with strong radial anchoring on the surface. In Theorem~\ref{thmrad} we show that any uniaxial equilibrium configuration must be spherically symmetric, in this particular nematic system. Such a result constitutes a first step towards a complete characterization of three-dimensional uniaxial equilibrium configurations. We expect the radial hedgehog to be the only non trivial uniaxial equilibrium.

Our main results, Theorem~\ref{thm2D} and Theorem~\ref{thmrad}, bring out the idea that the constraint of uniaxial symmetry is very restrictive and is in general not satisfied, except in very symmetric situations. These results
 shed a very new light on the phenomenon of `biaxial escape' \cite{sonnetkilianhess95}, and are \textit{fundamentally different} from the previous related ones in the literature. Indeed, biaxiality was always shown to occur by means of free energy comparison methods, while we \textit{only rely on the equilibrium equations}. In particular our results hold for all metastable configurations. Moreover, the appearance of biaxiality was usually related to special values of parameters such as the temperature \cite{mkaddemgartland00} -- which affects the bulk equilibrium --, or the size of the system \cite{bisigartlandrossovirga03} -- which affects the director deformation. We show instead that biaxiality occurs for \textit{any value of the temperature} (since the bulk energy density we work with is arbitrary) and \textit{any kind of director deformation.} In short: escape to biaxiality appears in all possible situations, and the equilibrium equations themselves force this escape.

The plan of the paper is the following. In Section~\ref{smod} we introduce the mathematical model describing orientational order. In Section~\ref{suniax} we derive the equilibrium equations for a configuration with uniaxial symmetry, and discuss the appearance of an extra equation corresponding to equilibrium with respect to symmetry-breaking perturbations. Sections \ref{s2D} and \ref{srad} contain the main results of the paper: in Section~\ref{s2D} we deal with one- and two-dimensional configurations and prove Theorem~\ref{thm2D}, and in Section~\ref{srad} we focus on a spherical nematic droplet with radial anchoring and prove Theorem~\ref{thmrad}.

\section{Description of the model}\label{smod}

\subsection{Order parameter and degrees of symmetry}\label{smodsorder}

In a nematic liquid crystal, the local state of alignment is described by an order parameter taking values in
\begin{equation}\label{S}
\mathcal S = \left\lbrace Q\in M_3(\mathbb R);\: Q={}^t\! Q,\: \mathrm{tr}\: Q =0 \right\rbrace,
\end{equation}
the set of all symmetric traceless $3\times 3$ matrices.

The group $G=SO(3)$ acts on the order parameter space $\mathcal S$: we denote by $\mathrm{Isom}(\mathcal S)$ the group of linear isometries of  $\mathcal S$, and the action is given by the group morphism
\begin{equation*}
\rho\colon G \to \mathrm{Isom}(\mathcal S),\quad \rho(g)Q=g Q {}^t\! g.
\end{equation*}
Note that this action $\rho$ is related to the natural action of $G$ on $\R^3$: $\rho(g) Q$ is the order parameter one should observe after changing the coordinate frame by $g$ in $\R^3$.

In the order parameter space $\mathcal S$ we may distinguish three types of elements, depending on their degree of symmetry. The degree of symmetry of an element $Q\in\mathcal S$ is given by its isotropy subgroup
\begin{equation*}
H(Q):= \left\lbrace g\in G,\:\rho(g) Q = Q \right\rbrace,
\end{equation*}
which can be of three different kinds:
\begin{itemize}
\item If $Q=0$, then $H(Q)=G$, and $Q$ describes the \emph{isotropic phase}.
\item If $Q$ has two equal (non zero) eigenvalues, then
\begin{equation*}
Q=\lambda\left( n\otimes n-\frac 13 \mathbf I \right),\quad\lambda\in\R^\ast,\: n\in\mathbb S^2,
\end{equation*}
and thus $Q=\lambda \rho(g) A_0$, where $A_0=\mathbf {e_z} \otimes \mathbf {e_z} -\mathbf I /3$ and $g\in G$ maps $\mathbf {e_z}$ to $n$. Therefore $H(Q)$ is conjugate via $g$ to 
\begin{equation*}
D_\infty := H(A_0) = \left\langle \lbrace r_{\mathbf {e_z},\theta }\rbrace_{\theta\in\R} , r_{\mathbf e_y,\pi} \right\rangle \approx O(2),
\end{equation*}
where $r_{n,\theta}$ stands for the element of $G$ corresponding to the rotation of axis $n$ and angle $\theta$. In this case, $Q$ describes the \emph{uniaxial phase}.
\item If $Q$ has three distinct eigenvalues, and $g\in G$ maps the canonical orthonormal basis $(\mathbf {e_x},\mathbf {e_y},\mathbf {e_z} )$ to an orthonormal basis of eigenvectors of $Q$, then $H(Q)$ is conjugate via $g$ to
\begin{equation*}
D_2 = \left\langle r_{\mathbf {e_x} ,\pi} , r_{\mathbf {e_y},\pi} \right\rangle \approx \mathbb Z / 2 \mathbb Z \times \mathbb Z / 2\mathbb Z.
\end{equation*}
In this case, $Q$ describes the \emph{biaxial phase}.
\end{itemize}

Hence there is a hierarchy in the breaking of symmetry that $Q$ can describe:
\begin{equation*}
\lbrace 0\rbrace \subset \mathcal U \subset \mathcal S,
\end{equation*}
where
\begin{equation}\label{Uset}
\mathcal U = \left\lbrace s \left( n\otimes n - \frac 13 \mathbf I \right);\: s\in\mathbb R,\: n\in \mathbb S^2 \right\rbrace,
\end{equation}
is the set of order parameter which can describe a breaking of symmetry from $G$ to $D_\infty$. Elements of $\mathcal U$ are characterized by their \emph{director} $n\in\mathbb S^2$ and their \emph{scalar order parameter} $s\in\mathbb R$.

\begin{rem}\label{remdetsn} Note that the scalar order parameter $s$ of a uniaxial tensor $Q\in\mathcal U$ is uniquely determined since $s=0$ if $Q=0$, and
\begin{equation*}
s=3\frac{\mathrm{tr}(Q^3)}{|Q|^2}
\end{equation*}
otherwise. On the other hand, the director is uniquely determined up to a sign if $Q\neq 0$, and not determined at all if $Q=0$.
\end{rem}

\subsection{Equilibrium configurations}\label{smodseq}

We consider a nematic liquid crystal contained in an open set $\Omega\subset\mathbb R^3$. The state of alignment of the material is described by a map
\begin{equation*}
Q \colon \Omega \to \mathcal S.
\end{equation*}
At equilibrium, the configuration should minimize a free energy functional of the form
\begin{equation*}
\mathcal F( Q) = \int_\Omega \left( f_{el} + f_b \right) dx, 
\end{equation*}
where $f_{el}$ is an elastic energy density, and $f_b$ is the bulk free energy.

Here we consider the one constant approximation for the elastic term:
\begin{equation*}
f_{el} = \frac{L}{2}|\nabla Q|^2,
\end{equation*}
and the most general frame invariant (i.e. invariant under the action $\rho$) bulk term:
\begin{equation*}
f_b =\varphi (\mathrm{tr}(Q^2),\mathrm{tr}(Q^3) ),
\end{equation*}
for some function
\begin{equation*}
\varphi \colon \R\times\R \to \R_+,
\end{equation*}
which we assume to be smooth.

\begin{rem}\label{remframeinv}
A fundamental property of the free energy density $f(Q)=f_{el}+f_b$ is its \emph{frame invariance}: for any $Q\in H^1_{loc}(\R^3;\mathcal S)$ it holds
\begin{equation*}
f(g\cdot Q)(x) = f(Q)(g^{-1}x)\qquad \forall g\in G,
\end{equation*}
where $g\cdot Q$ denotes the natural action of $G$ on maps $Q$, given by
\begin{equation}\label{framechange}
(g \cdot Q)(x)=\rho(g) Q(g^{-1}x)= g Q(g^{-1}x) g^{-1}.
\end{equation}
More general elastic terms $f_{el}$ are physically relevant, as long as the frame invariance property is conserved.
\end{rem}

An equilibrium configuration is described by a map $Q\in H^1(\Omega;\mathcal S)$ satisfying the Euler-Lagrange equation
\begin{equation}\label{EL}
L\Delta Q = 2(\partial_1\varphi) Q + 3 (\partial_2\varphi) \left(Q^2-\frac{|Q|^2}{3}I\right),
\end{equation}
associated to the free energy $\mathcal F$.

Physically relevant configurations should be bounded. On the other hand, classical elliptic regularity arguments ensure that any solution of \eqref{EL} which lies in $H^1\cap L^\infty$ is smooth. In fact, if in addition $\varphi$ is analytic, any $H^1\cap L^\infty$ solution of \eqref{EL} is actually analytic \cite[Theorem~6.7.6]{morrey}. 

In the sequel we will always consider smooth solutions. We discuss next a very mild sufficient condition on $\varphi$ which ensures boundedness -- and therefore smoothness -- of solutions.

In a bounded regular domain $\Omega$, a natural assumption on $\varphi$ which ensures that any $H^1$ solution of \eqref{EL} with bounded boundary data is in fact bounded is the following one:
\begin{equation}\label{growthvarphi}
\exists M>0 \text{ such that } \left( |Q|\geq M \Longrightarrow 2 |Q|^2(\partial_1\varphi) + 3(\partial_2\varphi)\mathrm{tr}(Q^3) \geq 0 \right).
\end{equation}
See \cite[Lemma~B.3]{lamy14} for a proof that assumption \eqref{growthvarphi} on $\varphi$ implies indeed that any $Q\in H^1$ solution of \eqref{EL} satisfies
\begin{equation*}
\|Q\|_{L^\infty(\Omega)} \leq \max (M,\| Q \|_{L^\infty(\partial\Omega)} ).
\end{equation*}

The fourth order approximation for $f_b$ usually considered in the literature
\begin{equation}\label{fblit}
f_b(Q)=-a\mathrm{tr}(Q^2)-b\mathrm{tr}(Q^3)+c\mathrm{tr}(Q^2)^2,
\end{equation}
corresponds to
\begin{equation*}
\varphi(x,y)=-ax -by + c x^2,
\end{equation*}
which satisfies indeed \eqref{growthvarphi}, as long as $c>0$ (and is obviously analytic).

\section{Uniaxial equilibrium}\label{suniax}

In the sequel, we investigate the existence of purely uniaxial equilibrium configurations, i.e. solutions $Q$ of the equilibrium equations \eqref{EL}, which satisfy
\begin{equation*}
Q (x) \in \mathcal U \qquad \forall x\in\Omega.
\end{equation*}
In other words, a purely uniaxial equilibrium configuration is a solution of \eqref{EL} which can be written in the form
\begin{equation}\label{ansatz}
Q(x)=s(x) \left( n(x)\otimes n(x) - \frac 13 \mathbf I \right),
\end{equation}
for some scalar field $s\colon\Omega\to\R$ and unit vector field $n\colon\Omega\to\mathbb S^ 2$.

\begin{rem}\label{remsnsmooth}
Here we do not require \textit{a priori} that the scalar field $s$ and the unit vector field $n$ in \textit{ansatz} \eqref{ansatz} be smooth. Note that $s$ is uniquely determined (see Remark~\ref{remdetsn}) by 
\begin{equation*}
s(x) = 3\frac{\mathrm{tr}(Q(x)^3)}{|Q(x)|^2}.
\end{equation*}
Therefore if $Q$ is smooth, then $s$ is smooth in the set $\lbrace Q\neq 0\rbrace \subset\Omega$ of points where $Q$ does not vanish, and continuous in $\Omega$. On the other hand, $n$ is not uniquely determined (see Remark~\ref{remdetsn}). However, in $\lbrace Q \neq 0 \rbrace$ one can choose locally a smooth unit vector field $n$. More precisely, if $Q$ is smooth and $x_0\in\Omega$ is such that $Q(x_0)\neq 0$, then there exists an open ball $B\subset\Omega$ centered at $x_0$, and a smooth map $n\colon B \to \mathbb S^2$ such that \eqref{ansatz} holds. The local smooth $n$ is obtained through the implicit function theorem (see the proof of Theorem~\ref{thm2D} below for more details). 
\end{rem}

\begin{rem}
Uniaxiality can be characterized through
\begin{equation*}
Q\in\mathcal U \Longleftrightarrow |Q|^6=6\left[\mathrm{tr}(Q^3)\right]^2,
\end{equation*}
so that any analytic map $Q\colon \Omega\to\mathcal S$ which is uniaxial in some
open subset of $\Omega$ is automatically uniaxial everywhere \cite{majumdarzarnescu10}. Thus, for analytic $\varphi$, Theorems~\ref{thm2D} and \ref{thmrad} proved below are valid if we replace the assumption that $Q$ be purely uniaxial, with the assumption that $Q$ be uniaxial in some open set.
\end{rem}

\begin{rem}\label{remhedgehog}
The spherically symmetric radial hedgehog \cite{mkaddemgartland00} provides an example of purely uniaxial equilibrium (see also Section~\ref{srad} below). However, in the particular case of the radial hedgehog, uniaxial symmetry is a consequence of spherical symmetry, for which Palais' Principle of Symmetric Criticality applies \cite{palais79}. 
The Principle of Symmetric Criticality is a general tool which allows to prove existence of symmetric equilibria. Roughly speaking, if the free energy and the space of admissible configurations are `symmetric', then the Principle asserts the following: any symmetric configuration which is an equilibrium with respect to symmetry-preserving perturbations is \emph{automatically} an equilibrium with respect to symmetry-breaking perturbations also. Of course the meaning of `symmetric' needs to be precised: see \cite{palais79} for a rigorous mathematical framework in which this Principle is valid.
\end{rem}

However, in general the Principle of Symmetric Criticality does not apply to uniaxial symmetry, as is suggested by the following result (see Remark~\ref{remcritsym} below).

\begin{prop}\label{ELuniax}
Let $\omega\subset \mathbb R^3$ be an open set. Let $s\colon\omega\to\R$ and $n\colon\omega\to\mathbb S^2$ be smooth maps such that the corresponding uniaxial $Q$ \eqref{ansatz} satisfies the equilibrium equation \eqref{EL}. Then $s$ and $n$ satisfy
\begin{equation}\label{U}
\left\lbrace
\begin{gathered}
\Delta s = 3 |\nabla n|^2 s +\frac{1}{L}(2s\partial_1\varphi + s^2 \partial_2\varphi), \\
s\Delta n + 2 (\nabla s \cdot\nabla)n = -s|\nabla n|^2 n,
\end{gathered}
\right.
\end{equation}
and, in regions where $s$ does not vanish, $n$ satisfies the extra equation
\begin{equation}\label{extra}
2 \sum_{k=1}^3 \partial_k n \otimes \partial_k n  = |\nabla n|^2 \left(I-n \otimes   n\right).
\end{equation}
\end{prop}
\begin{proof}
Plugging the uniaxial \textit{ansatz} \eqref{ansatz} into the equilibrium equation \eqref{EL}, we find, after rearranging the terms,
\begin{equation*}
M_1+M_2+M_3=0,
\end{equation*}
where
\begin{align*}
M_1 & = \left[ \Delta s - 3|\nabla n|^2 s -\frac{1}{L}(2s\partial_1\varphi + s^2 \partial_2\varphi)\right]\left(n\otimes n-\frac{1}{3} \mathbf I\right), \\
M_2 & =2n \odot (s\Delta n + 2(\nabla s\cdot \nabla)n + s|\nabla n|^2 n),\\
M_3 & = s \left[ 2\sum_k \partial_k n \otimes\partial_k n + |\nabla n|^2 \left(  n\otimes n-\mathbf I\right) \right].
\end{align*}
Here $\odot$ denotes the symmetric tensor product: the $(i,j)$ component of $n \odot m$ is $(n_i m_j+n_j m_i)/2$.

Using the fact that $|n|^2$ is constant equal to 1, which implies in particular $n\cdot\partial_j n=0$ and $n\cdot\Delta n + |\nabla n|^2=0$, we find that
\begin{align*}
M_1 & \in \mathcal S \cap \mathrm{Span} \left( n\odot n-\frac 13 \mathbf I\right),\\
M_2 & \in \mathcal S \cap \mathrm{Span} \left\lbrace n\odot v \colon v\in n^\perp\right\rbrace,\\
M_3 & \in \mathcal S \cap \mathrm{Span} \left \lbrace v \odot w \colon v,w\in n^\perp \right\rbrace.
\end{align*}
Recall here that $\mathcal S$ is the order parameter space \eqref{S} of traceless symmetric matrices.
In particular, $M_1$, $M_2$ and $M_3$ are pairwise orthogonal, and we deduce that
\begin{equation*}
M_1=M_2=M_3=0.
\end{equation*}
We conclude that \eqref{U} and \eqref{extra} hold.
\end{proof}

\begin{rem}\label{remcritsym}
The system \eqref{U} satisfied by $(s,n)$ is nothing else than the Euler-Lagrange equation associated to the energy
\begin{equation*}
F(s,n) = \mathcal F(Q) = \int \left[\frac{L}{2}\left(\frac{2}{3}|\nabla s|^2 + 2s^2|\nabla n|^2\right)+\varphi(2s^2/3,2s^3/9) \right] dx,
\end{equation*}
under the constraint $|n|^2=1$.
In other words \eqref{U} expresses the fact that $Q$ is an equilibrium of $\mathcal F$ with respect to perturbations preserving the symmetry constraint $Q\in \mathcal U$. The minimization of the functional $F$ has been studied in \cite{lin91}. On the other hand, the extra equation \eqref{extra} expresses the fact that $Q$ is an equilibrium with respect to symmetry-breaking perturbations. Since \eqref{extra} is not trivial, we see that Palais' Principle of Symmetric Criticality does not apply to uniaxial symmetry.
\end{rem}

\begin{rem}\label{remextra}
The extra equation \eqref{extra} is of the form $M_3=0$, with $M_3$ taking its values in $\mathcal S$ of dimension 5: it contains 5 scalar equations. However, it has been shown during the proof of Proposition~\ref{ELuniax} that, due to the constraint $n\in\mathbb S^2$, it holds in fact
\begin{equation*}
M_3 \in \mathcal M := \mathcal S \cap \mathrm{Span} \left\lbrace v \otimes w \colon v,w\in n^\perp \right\rbrace.
\end{equation*}
Since $\mathcal M$ and $\mathbb S^2$ are two-dimensional, the information really carried by \eqref{extra} corresponds to a system of two first order partial differential equations, with two unknown. In particular, given a generic non characteristic boundary data $n_{|\partial\omega}=n_0$, equation \eqref{extra} should have exactly one local solution. Therefore, system \eqref{U} coupled with Dirichlet boundary conditions and the extra equation \eqref{extra} is strongly overdetermined. We expect solutions to exist only in very `symmetric' cases. The results presented in the sequel are indeed of such a nature.
\end{rem}

\section{In one and two dimensions}\label{s2D}

In this section we concentrate on one- and two-dimensional configurations, which occur in case of translational invariance in at least one direction.
Such a symmetry assumption is actually relevant for many nematic systems that are interesting both theoretically and for application purposes. For instance, in nematic cells bounded by two parallel plates with competing anchoring, one usually looks for one-dimensional solutions \cite{palffymuhoraygartlandkelly94, bisigartlandrossovirga03,ambrozicbisivirga08}. Such hybrid nematic cells provide a model system for understanding the physics of frustration, and this kind of geometry occurs in several nematic based optical devices. 
Another relevant geometry is the cylindrical one, in which two dimensional configurations can be considered \cite{ kraljvirgazumer99, sonnetkilianhess95,lucarey07a,lucarey07b}, with applications to high performance fibers \cite{cheongrey04,chanetal05,jianhurtsheldoncrawford06}.

Our conclusion (see Theorem~\ref{thm2D} below) is that a one- or two-dimensional equilibrium configuration can be purely uniaxial only if the director field is constant. Thus in the translation-invariant case, the system \eqref{U} coupled with \eqref{extra} is so strongly overdetermined that it admits only trivial solutions.

\begin{thm}\label{thm2D}
Let $\Omega\subset\R^3$ be an open set and $Q$ be a smooth solution of the equilibrium equation \eqref{EL}. Assume that $Q$ is invariant in one direction: there exists $\nu_0\in\mathbb S^2$ such that $\nu_0 \cdot\nabla Q \equiv 0$. 

\begin{itemize}

\item[(i)] If $Q$ is purely uniaxial (i.e. takes values in $\mathcal U$) then $Q$ has constant director in every connected component of $\lbrace Q\neq 0\rbrace$. That is, for every connected component $\omega$ of $\lbrace Q\neq 0\rbrace$, there exists a uniform director $n_0=n_0(\omega)\in\mathbb S^2$ such that
\begin{equation*}
Q(x)=s(x)\left( n_0\otimes n_0 - \frac 13 \mathbf I \right)\qquad\forall x\in\omega.
\end{equation*}
for some scalar vector field $s\colon\Omega\to\R$.

\item[(ii)] If in addition $Q$ is analytic and $\Omega$ is connected, then the director is the same in every connected component of $\lbrace Q\neq 0\rbrace$: there exists $n_0\in\mathbb S^2$ such that
\begin{equation*}
Q(x)=s(x)\left( n_0\otimes n_0 - \frac 13 \mathbf I \right)\qquad\forall x\in\Omega.
\end{equation*}

\end{itemize}
\end{thm}

\begin{rem}\label{1D}
The one-dimensional case is of course contained in the two-dimensional one, but we find useful to present a specific, much simpler argument here. In one dimension the extra equation \eqref{extra} becomes
\begin{equation*}
2 n' \otimes n'  = |n'|^2 (\mathbf I - n\otimes n),
\end{equation*} 
which readily implies $n'\equiv 0$. Thus in one dimension the conclusion of Theorem~\ref{thm2D} is achieved using only the extra equation \eqref{extra}.

In two dimensions however, the proof of Theorem~\ref{thm2D} is more involved. In particular, the extra equation \eqref{extra} does admit non trivial solutions. For instance a cylindrically symmetric director field introduced by Cladis and Kl\'e{}man \cite{cladiskleman72} and studied further in \cite{bethuelbreziscolemanhelein92}, which is given in cylindrical coordinates by
\begin{equation*}
n(r,\theta,z)=\cos \psi(r) \:\mathbf{e_r} + \sin \psi(r) \:\mathbf{e_z}\quad\text{with }r\frac{d\psi}{dr} = \cos\psi,
\end{equation*}
satisfies \eqref{extra}. But there cannot exist any scalar field $s$ such that $(s,n)$ solves \eqref{U}.
\end{rem}

\begin{proof}[Proof of Theorem~\ref{thm2D}]
Since the free energy density is frame invariant (see Remark~\ref{remframeinv}) we may assume that $\nu_0=\mathbf{e_z}$, so that $\partial_3 Q \equiv 0$.

We start by proving assertion \textit{(i)} of Theorem~\ref{thm2D}. Fix a connected component $\omega$ of $\lbrace Q\neq 0 \rbrace$ and define the smooth map $s\colon\omega\to\R$ by the formula
\begin{equation*}
s= 3\frac{\mathrm{tr}(Q^3)}{|Q|^2}.
\end{equation*}
Recall that $s(x)$ is the scalar order parameter of $Q(x)\in\mathcal U$ (see Remark~\ref{remdetsn}). In particular, $s$ does not vanish in $\omega$. In the sequel we are going to show that the smooth map $Q/s$ is locally constant in $\omega$, which obviously implies \textit{(i)}.

Let $x_0\in\omega$. We claim that there exists an open ball $B\subset\omega$ centered at $x_0$ and a smooth map $n\colon B\to\mathbb S^2$ such that the formula for $Q$ in terms of $s$ and $n$ \eqref{ansatz} holds in  $B$ (as announced in Remark~\ref{remsnsmooth}). 

Indeed, fix a director $n_0\in\mathbb S^2$ of $Q(x_0)$: it holds $Q(x_0)n_0=s_0 n_0$, with $s_0=s(x_0)$. Since the eigenvalue $s_0$ is simple and $Q(x_0)$ maps $n_0^\perp$ to $n_0^\perp$, the implicit function theorem can be applied to the map
\begin{equation*}
\omega \times \R \times n_0^\perp \to \R^3,\quad (x,s,v)\mapsto \left( Q(x)-s \right)(n_0+v)
\end{equation*}
to obtain smooth maps $v$ and $\tilde s$ defined in a neighborhood of $x_0$ and solving uniquely
\begin{equation*}
Q(x)(n_0+v) = \tilde s (n_0 + v)\quad\text{for }\tilde s\approx s_0,\; v\approx 0\in n_0^\perp.
\end{equation*}
Since, for $x$ close enough to $x_0$, eigenvalues of $Q(x)$ distinct from $s(x)$ are far from $s_0$, it must hold $\tilde s = s$. Therefore $n=(n_0+v)/|n_0+v|$ provides a smooth map such that \eqref{ansatz} holds in a neighborhood of $x_0$, which we may assume to be an open ball $B$.

To prove \textit{(i)} it remains to show that $n$ is constant in $B$, which obviously implies that $Q/s$ is locally constant (since $x_0\in \omega$ is arbitrary).

We start by noting that, since by assumption $\partial_3 Q=0$, it holds 
\begin{equation*}
\partial_3 s = \frac 32 n\cdot (\partial_3 Q)n = 0,\quad
\partial_3 n =\frac{1}{s} (\partial_3 Q) n = 0.
\end{equation*}
Thus \eqref{extra} becomes
\begin{equation*}
A:=2\partial_1 n \otimes \partial_1 n + 2 \partial_2 n \otimes \partial_2 n 
- \left( |\partial_1 n|^2 + |\partial_2 n|^2\right) (I-n\otimes n) = 0.
\end{equation*}
We deduce that
\begin{equation*}
\partial_1 n \cdot A\partial_2 n = |\nabla n|^2 \partial_1 n \cdot \partial_2 n = 0,
\end{equation*}
which implies
\begin{equation}\label{2D1}
\partial_1 n \cdot \partial_2 n = 0\quad\text{in }B.
\end{equation}
Using this last fact, we compute
\begin{align*}
\partial_1 n \cdot A \partial_1 n & = |\partial_1 n|^2 \left(|\partial_1 n|^2-|\partial_2 n|^2 \right)=0, \\
\partial_2 n \cdot A \partial_2 n & = |\partial_2 n|^2\left(|\partial_2 n|^2 - |\partial_1 n|^2 \right) = 0,
\end{align*}
from which we infer
\begin{equation}\label{2D2}
|\partial_1 n|^2 = |\partial_2 n|^2\quad\text{in }B.
\end{equation}
As a consequence of \eqref{2D1} and \eqref{2D2}, we obtain that
\begin{align*}
\Delta n \cdot \partial_1 n  & = \frac 12 \partial_1 \left[ |\partial_1 n|^2 - |\partial_2 n|^2 \right] + \partial_2 \left[\partial_1 n\cdot\partial_2 n \right] = 0,  \\
\Delta n \cdot \partial_2 n & = \frac 12 \partial_2 \left[ |\partial_2 n|^2 - |\partial_1 n|^2 \right] + \partial_1 \left[\partial_1 n\cdot\partial_2 n \right] = 0.
\end{align*}
That is, the vector $\Delta n$ is orthogonal to both vectors $\partial_1 n$ and $\partial_2 n$.
Therefore, taking the scalar product of the second equation of \eqref{U} with $\partial_1 n$ and $\partial_2 n$, we are left with
\begin{equation}\label{2D3}
\partial_1 s |\nabla n|^2 = \partial_2 s |\nabla n|^2 = 0 \quad\text{in }B.
\end{equation}
We claim that \eqref{2D3} implies in fact 
\begin{equation}\label{2D4}
|\nabla n|^2 = 0 \quad\text{in }B.
\end{equation}
Assume indeed that \eqref{2D4} does not hold, so that $|\nabla n|^2>0$ in some open set $W\subset B$. Then by \eqref{2D3} the scalar field $s$ is constant in $W$, and the first equation of \eqref{U} implies that $|\nabla n|^2$ is constant in $W$. Up to rescaling the variable, we have thus obtained a map $n$ mapping an open subset of the plane $\R^2$ into the sphere $\mathbb S^2$ and satisfying
\begin{equation*}
\partial_1 n \cdot \partial_2 n = 0,\quad |\partial_1 n|^2 = |\partial_2 n|^2=1.
\end{equation*}
That is, $n$ is a local isometry. Since the plane has zero curvature while the sphere has positive curvature, the existence of such an isometry contradicts Gauss's \emph{Theorema egregium}. Hence we have proved the claim \eqref{2D4}, and $n$ must be constant in $B$. This ends the proof of \textit{(i)}.

Now we turn to the proof of assertion \textit{(ii)} of Theorem~\ref{thm2D}. We start by proving the following

\vspace{.2em}
\textit{Claim:} for any open ball $B\subset\Omega$, $Q$ has constant director in $B$: there exists $n_0\in\mathbb S^2$ such that 
$Q=s(n_0\otimes n_0 -\mathbf I /3)$ in $B$.
\vspace{.2em}

Note that this \textit{Claim} is simply a consequence of \textit{(i)} if $B\subset\lbrace Q\neq 0\rbrace$. The additional information here is that $B\cap \lbrace Q\neq 0\rbrace$ may not be connected.

If $Q\equiv 0$ in $B$ the \textit{Claim} is obvious, so we assume $Q(x_0)\neq 0$ for some $x_0\in B$. Let $n_0\in\mathbb S^2$ be such that
\begin{equation}\label{2D5}
Q(x_0)=s(x_0)\left( n_0\otimes n_0 -\frac 13 \mathbf I \right).
\end{equation}
We now prove the \textit{Claim} by contradiction: assume that there exists $x_1\in B$ such that
\begin{equation}\label{2D6}
Q(x_1)\neq s(x_1)\left( n_0\otimes n_0 -\frac 13 \mathbf I \right).
\end{equation}
In particular, $Q(x_1)\neq 0$. Consider the segment $S=[x_0,x_1]$ contained in $B$ and therefore in $\Omega$. Since $Q$ is analytic and does not vanish identically on $S$, the set $S\cap \lbrace Q=0\rbrace$ must be discrete (and thus finite by compactness). 
%On the other hand, $Q$ must vanish on $S$: else $Q$ would have constant director on $S$ by assertion \textit{(i)}, contradicting \eqref{2D6}. 

Since \eqref{2D6} holds, the (locally constant) director is not the same in the respective connected components of $x_0$ and $x_1$ in $S\cap \lbrace Q\neq 0\rbrace$. As a consequence, there must exist $x_2\in S$, $n_1\in\mathbb S^2\setminus \lbrace\pm n_0\rbrace$ and $\delta>0$ such that:
\begin{gather*}
\lbrace Q=0\rbrace \cap S \cap B_\delta(x_2) = \lbrace x_2 \rbrace,\\
Q(x)=s(x)\left( n_0 \otimes n_0 -\frac 13 \mathbf I \right) \quad \forall x\in [x_2,x_0]\cap B_\delta (x_2),\\
Q(x)=s(x)\left( n_1 \otimes n_1 -\frac 13 \mathbf I \right) \quad \forall x\in [x_2,x_1]\cap B_\delta (x_2).
\end{gather*}
Hence for small enough $\varepsilon$, the analytic map
\begin{equation*}
\widetilde Q\colon (-\varepsilon,\varepsilon)\ni t \mapsto Q(x_2 + t (x_0-x_1)) \in\mathcal U
\end{equation*}
vanishes exactly at $t=0$, has constant director $n_0$ for $t>0$ and constant director $n_1$ for $t<0$. The associated map $\tilde s(t)$ is smooth in $(-\varepsilon,\varepsilon)\setminus\lbrace 0\rbrace$ and it holds
\begin{equation*}
\widetilde Q'(t)=\begin{cases}
\tilde s'(t) \left( n_0\otimes n_0 -\frac 13 \mathbf I \right) & \text{for }t> 0\\
\tilde s'(t) \left( n_1\otimes n_1 -\frac 13 \mathbf I \right) & \text{for }t< 0.
\end{cases}
\end{equation*}
We deduce that $l^+ := \lim_{0^+} \tilde s'$ and $l^-:=\lim_{0^-} \tilde s'$ exist and satisfy
\begin{equation*}
\widetilde Q'(0) = l^+ \left( n_0\otimes n_0 -\frac 13 \mathbf I \right) = l^- \left( n_1\otimes n_1 -\frac 13 \mathbf I \right).
\end{equation*}
Since $n_0\neq \pm n_1$, it must hold $l^+=l^-=0$. Thus $\tilde s$ is in fact $C^1$ in $(-\varepsilon,\varepsilon)$ and satisfies $\tilde s'(0)=0$.

For any integer $k\geq 0$ it holds
\begin{equation*}
\widetilde Q^{(k)}(t)=\begin{cases}
\tilde s^{(k)}(t) \left( n_0\otimes n_0 -\frac 13 \mathbf I \right) & \text{for }t > 0\\
\tilde s^{(k)}(t) \left( n_1\otimes n_1 -\frac 13 \mathbf I \right) & \text{for }t < 0,
\end{cases}
\end{equation*}
and we may repeat the same argument as above to show by induction that $\tilde s$ is smooth in $(-\varepsilon,\varepsilon)$ and all its derivatives vanish at 0. In particular we find that
\begin{equation*}
\widetilde Q^{(k)}(0)=0 \qquad\forall k\geq 0,
\end{equation*}
which implies that $Q\equiv 0$ on $S$, since $\widetilde Q$ is analytic: we obtain a contradiction, and the above \textit{Claim} is proved.

We may now complete the proof of assertion \textit{(ii)} of Theorem~\ref{thm2D}. Let $x,y\in\Omega$ and $n_x,n_y\in\mathbb S^2$ be such that
\begin{equation*}
Q(x)=s(x)\left(n_x\otimes n_x -\frac 13 \mathbf I \right),\quad Q(y)=s(y)\left(n_y\otimes n_y - \frac 13 \mathbf I \right).
\end{equation*}
We aim at showing that $n_x$ and $n_y$ can be choosen so that $n_x=\pm n_y$. In particular we may assume that $Q(x)\neq 0$ and $Q(y)\neq 0$.

Since $\Omega$ is open and connected (and thus path-connected), there exists a ``chain of open balls'' from $x$ to $y$. More explicitly: there exist points
\begin{equation*}
x=x_0,x_1,\ldots,x_{N-1},x_N=y \in\Omega,
\end{equation*}
and open balls
\begin{equation*}
B_0\ni x_0, B_1\ni x_1 , \ldots , B_N \ni x_N,
\end{equation*}
such that
\begin{equation*}
B_k \cap B_{k+1} \neq \emptyset \qquad k = 0,\ldots, N_1.
\end{equation*}
We denote by $n_k\in\mathbb S^2$ a constant director in $B_k$, provided by the above \textit{Claim}. In particular $n_0=\pm n_x$ and $n_N=\pm n_y$. 

In the intersection $B_k\cap B_{k+1}$, $n_k$ and $n_{k+1}$ are both admissible constant directors. Since $Q$ is analytic and not uniformly zero, it can not be uniformly zero in the non empty open set $B_k\cap B_{k+1}$, and we deduce that it must hold $n_k = \pm n_{k+1}$. Therefore we find
\begin{equation*}
n_0 = \pm n_1 = \pm n_2 = \cdots = \pm n_N,
\end{equation*}
and in particular $n_x=\pm n_y$. The proof of \textit{(ii)} is complete.
\end{proof}

\begin{rem}\label{remQsmooth}
As already pointed out in Section~\ref{smodseq}, the assumption that $Q$ is smooth is very natural, since physically relevant solutions are bounded and therefore smooth. The additional assumption of analyticity in assertion \textit{(ii)} is also natural, since it is satisfied whenever the bulk free energy is analytic (and this is the case for the bulk free energy usually considered).
\end{rem}

\section{In a spherical droplet with radial anchoring}\label{srad}

In this section we consider a droplet of nematic subject to strong radial anchoring on the surface. Droplets of nematic  play an important role in some electro-optic applications, like polymer dispersed liquid crystals (PDLC) devices (see the review article \cite{lopezleon11} and references therein). Moreover, this problem is important theoretically as a model problem for the study of point defects, due to the universal features it exhibits \cite{kraljvirga01}.

The droplet containing the nematic is modelled as an open ball
\begin{equation*}
B_R = \left\lbrace x\in\R^3 \colon |x|<R \right\rbrace,
\end{equation*}
and strong radial anchoring corresponds to Dirichlet boundary conditions of the form
\begin{equation}\label{bc}
Q(x)=s_0\left(\frac x R \otimes \frac x R - \frac 13 I \right) \qquad\text{for }|x|=R,
\end{equation}
for some fixed $s_0\neq 0$.

In this setting, the equilibrium equation \eqref{EL} admits a particular symmetric solution of the form
\begin{equation}\label{radialQ}
Q(x)=s(r)\left(\frac x r \otimes \frac x r - \frac 13 I \right) \qquad\forall x\in B_R,
\end{equation}
where $r=|x|$, and $s:(0,R)\to\mathbb R$ solves
\begin{equation}\label{sradial}
\frac{d^2s}{dr^2} + \frac{2}{r}\frac{ds}{dr}-\frac{6}{r^2}s = \frac{1}{L}\big(2s\partial_1\varphi(2s^2/3,2s^3/9) +s^2\partial_2\varphi(2s^2/3,2s^3/9) \big),
\end{equation}
with boundary conditions $s(0)=0$, $s(R)=s_0$. We call such a solution \emph{radial hedgehog}.

As already mentioned in Remark~\ref{remhedgehog}, the existence of such a solution is ensured by Palais' Principle of Symmetric Criticality \cite{palais79}. In fact, $G=SO(3)$ acts linearly and isometrically on the affine Hilbert space
\begin{equation*}
\mathcal H = \left\lbrace Q\in H^1(B_R;\mathcal S) \colon Q \text{ satisfies }\eqref{bc}\right\rbrace
\end{equation*}
by change of frame: the action is given by formula \eqref{framechange}. The free energy is frame invariant (see Remark~\ref{remframeinv}): it holds
\begin{equation*}
\mathcal F (g\cdot Q) = \mathcal F (Q)\qquad\forall g\in G,\; Q\in\mathcal H.
\end{equation*}
Denoting by $\Sigma \subset \mathcal H$ the subspace of symmetric configurations, i.e. of those maps $Q$ which satisfy $g\cdot Q = Q$ for all rotations $g\in G$, the Principle of Symmetric Criticality \cite[Section~2]{palais79} can be stated as follows: if $Q\in\Sigma$ is a critical point of $\mathcal F_{|\Sigma}$, then $Q$ is a critical point of $\mathcal F$, i.e. $Q$ solves the equilibrium equation \eqref{EL}.

Since $\Sigma$ consists precisely of those $Q$ which are of the form \eqref{radialQ}, and since the existence of a minimizer of $\mathcal F_{|\Sigma}$ is ensured by the direct method of the calculus of variations \cite{lamy13}, we obtain the existence of the radial hedgehog solution of \eqref{EL} described above by \eqref{radialQ}-\eqref{sradial}. 

Spherically symmetric solutions are in fact the only purely uniaxial solutions of this problem. This is the content of the next result.

\begin{thm}\label{thmrad}
Assume that $\varphi$ is analytic and satisfies \eqref{growthvarphi}.
Let $Q\in H^1(B_R,\mathcal S)$ solve the equilibrium equation \eqref{EL}, with radial boundary conditions \eqref{bc}.

If $Q$ is purely uniaxial (i.e. takes values in $\mathcal U$), then $Q$ is necessarily spherically symmetric: it satisfies \eqref{radialQ}-\eqref{sradial}.
\end{thm}

\begin{rem}  A recent result of Henao and Majumdar \cite{henaomajumdar12,henaomajumdar13} is a direct corollary of Theorem~\ref{thmrad}. In \cite{henaomajumdar12,henaomajumdar13}, the authors   consider a spherical droplet with radial anchoring, with a bulk free energy $f_b$ of the form \eqref{fblit} and study the low temperature limit $a\to\infty$. They assume the existence of a sequence of uniaxial minimizers of the free energy, and show convergence towards a spherically symmetric solution.
\end{rem}

\begin{proof}[Proof of Theorem~\ref{thmrad}:]
The assumption \eqref{growthvarphi} on $\varphi$ ensures that $Q$ is bounded and therefore analytic (see Section~\ref{smodseq}).

Since $Q$ is smooth up to the boundary $\partial B$, and does not vanish on the boundary, we may proceed as in the proof of Theorem~\ref{thm2D} to obtain, in a neighborhood of each point of the boundary $\partial B_R$, smooth maps $s$ and $n$ such that the \textit{ansatz} \eqref{ansatz} holds (see also Remark~\ref{remsnsmooth}). The locally well-defined map $n$ is determined up to a sign. We determine it uniquely via the boundary condition
\begin{equation*}
n(x)=\frac x R \qquad \text{for }|x|=R.
\end{equation*}
Therefore we obtain, for some $\delta >0$, smooth maps
\begin{equation*}
s\colon B_R\setminus B_{(1-\delta)R} \to \mathbb R,\quad n\colon B_R\setminus B_{(1-\delta)R}\to\mathbb S^2,
\end{equation*}
such that
\begin{equation*}
Q(x)=s(x)\left( n(x)\otimes n(x)-\frac 13 I \right) \quad\text{for }(1-\delta)R < |x|<R.
\end{equation*}
The values of $s$ and $n$ on the boundary $\partial B_R$ are determined:
\begin{equation}\label{bcsn}
s(x)=s_0,\quad n(x)=\frac x R \qquad\text{for }|x|=R.
\end{equation}
We use the fact that $s$ and $n$ satisfy the system \eqref{U} and the extra constraint \eqref{extra}, to determine in addition their radial derivatives on the boundary:

\begin{lem}\label{lemradderiv}
It holds
\begin{equation*}
\partial_r n \equiv 0,\quad \partial_r s \equiv s_1,\qquad\text{on }\partial B_R,
\end{equation*}
for some constant $s_1\in\mathbb R$.
\end{lem}

Lemma~\ref{lemradderiv} constitutes the heart of the proof of Theorem~\ref{thmrad}.
The proof of Lemma~\ref{lemradderiv} can be found below. We start by showing
 how Lemma~\ref{lemradderiv} implies the conclusion of Theorem~\ref{thmrad}. 
 
Let $\tilde s$ be a local solution of \eqref{sradial} with Cauchy data
\begin{equation*}
\tilde s(R)=s_0,\; \frac{d\tilde s}{dr}(R)=s_1,
\end{equation*}
where $s_1$ is the constant value of $\partial_r s$ on $\partial B_R$ according to Lemma~\ref{lemradderiv}. We fix $\eta>0$ such that $\tilde s$ is defined on $[R,R+\eta]$, and define a map $\widetilde Q$ on $B_{R+\eta}$ by
\begin{equation*}
\widetilde Q (x) = \begin{cases}
Q(x) & \text{if }|x|\leq R, \\
\tilde s(r) \left( \frac x r \otimes \frac x r- \frac 13 I\right) & \text{if }R<|x|<R+\eta.
\end{cases}
\end{equation*}
Lemma~\ref{lemradderiv} ensures that the boundary conditions on $\partial B_R$ match well at order 0 and 1:
the map $\widetilde Q$ belongs to $C^1(\overline B_{R+\eta})$. Moreover, the matching boundary conditions on $\partial B_R$ ensure that $\widetilde Q$ is a weak solution of the Euler-Lagrange equation \eqref{EL} in $B_{R+\eta}$. In particular, $\widetilde Q$ is analytic (see Section~\ref{smodseq}). Hence, for any rotation $g\in G$, the map 
\begin{equation*}
x\mapsto \widetilde Q (gx)- g \widetilde Q(x) {}^t\! g
\end{equation*}
is analytic and vanishes in $B_{R+\eta}\setminus B_R$ and must therefore vanish everywhere. We deduce that $Q$ is spherically symmetric and the proof of Theorem~\ref{thmrad} is complete.
\end{proof}

\begin{proof}[Proof of Lemma~\ref{lemradderiv}:]
During this proof we make  use of spherical coordinates $(r,\theta,\varphi)$ and denote by $(\mathbf{e_r},\mathbf{e_\theta},\mathbf{e_\varphi})$ the associated (moving) eigenframe. 

For simplicity we assume $R=1$ (the general case follows by rescaling the variable) and write $B$ for $B_1$. We proceed in three steps: we start by showing that, on the boundary $\partial B$, it holds 
\begin{itemize}
\item $\partial_r n = 0$, 
\item then $\partial_r^2n = 0$, 
\item and eventually $\partial_\theta \partial_r s = \partial_\varphi \partial_r s = 0$.
\end{itemize}

\vspace{.2em}
\textit{Step 1:} $\partial_r n=0$ on $\partial B$.

This first step is obtained as a consequence of the boundary condition \eqref{bcsn}, and of the constraint \eqref{extra}. Indeed, on the boundary, \eqref{bcsn} determines the partial derivatives of $n$ in two directions $\partial_\theta n$ and $\partial_\varphi n$, and \eqref{extra} determines the partial derivative in the remaining direction.

In spherical coordinates, \eqref{extra} becomes
\begin{equation}\label{extraspheric}
2\left( \partial_r n\otimes \partial_r n + \frac{1}{r^2}\partial_\theta n\otimes\partial_\theta n + \frac{1}{r^2\sin^2\theta}\partial_\varphi n \partial_\varphi n \right) = |\nabla n|^2 (I-n\otimes n),
\end{equation}
and 
\begin{equation*}
|\nabla n|^2 = |\partial_r n|^2 + \frac{1}{r^2}|\partial_\theta n|^2 + \frac{1}{r^2\sin^2\theta}|\partial_\varphi n|^2.
\end{equation*}
Since on the boundary $\partial B$ it holds
\begin{equation*}
n=\mathbf{e_r},\;\partial_\theta n = \mathbf{e_\theta},\; \partial_\varphi n = \sin\theta \mathbf{e_\varphi},
\end{equation*}
we deduce from \eqref{extraspheric} that
\begin{equation*}
2\partial_r n\otimes\partial_r n = |\partial_r n|^2(I-\mathbf{e_r}\otimes \mathbf{e_r})\quad\text{for }r=1,
\end{equation*}
which implies $\partial_r n=0$ and proves \textit{Step 1}.

\vspace{.2em}
\textit{Step 2:} $\partial_r^2 n=0$ on $\partial B$.

This second step is obtained as a consequence of \textit{Step 1} and of the second equation of \eqref{U}, together with the boundary conditions \eqref{bcsn}. In fact, it holds
\begin{equation*}
\Delta n = \partial_r^2 n  + \Delta_{\mathbb S^2} n = \partial_r^2 n -2\mathbf{e_r} \quad\text{for }r=1,
\end{equation*}
since $\partial_r n=0$ by \textit{Step 1} and $n=\mathbf{e_r}$ for $r=1$. Moreover, since $s$ is constant on the boundary, it holds
\begin{equation*}
(\nabla s \cdot)n = \partial_r s \partial_r n = 0\quad\text{for }r=1.
\end{equation*}
Thus the second equation of \eqref{U} becomes, on the boundary,
\begin{equation*}
s_0\partial_r^2 n -2 s_0 \mathbf{e_r} = - s_0 |\nabla n|^2 \mathbf{e_r} = -2 s_0 \mathbf{e_r} \quad\text{for }r=1.
\end{equation*}
Here we used again \eqref{bcsn} and \textit{Step 1} to compute $|\nabla n|^2$ for $r=1$. The last equation completes the proof of \textit{Step 2}.

\vspace{.2em}
\textit{Step 3:} $\partial_\theta\partial_r s = \partial_\varphi\partial_r s= 0$ on $\partial B$.

To prove this third step, we consider Taylor expansions of $s$ and $n$ with respect to $r-1\approx 0$, and plug them into \eqref{U} and \eqref{extra} to obtain more information about higher order radial derivatives and find eventually that $\partial_r s$ is constant on the boundary.

Using \textit{Step 1} and \textit{Step 2}, we may write, for $r=|x|\in [1-\delta, 1]$ and $\omega=x/r \in \mathbb S^2$,
\begin{align}
n &=\mathbf{e_r} + (r-1)^ 3 m_1(\omega) + (r-1)^4m_2(\omega) + O((r-1)^5)\label{taylorn}\\
s&=s_0+(r-1)s_1(\omega)+(r-1)^2s_2(\omega)+O((r-1)^3),\label{taylors}
\end{align}
where $6m_1=\partial_r^3 n_{|\mathbb S^2}$, $24m_2 = \partial_r^4 n_{|\mathbb S^2}$, $s_1=\partial_r s_{|\mathbb S^2}$, and $2s_2=\partial_r^2 s_{|\mathbb S^2}$ are smooth functions of $\omega\in\mathbb S^2$, and
\begin{equation*}
O((r-1)^k)=(r-1)^k \times \text{ some smooth function of }(r,\omega).
\end{equation*}

In the sequel, we plug the Taylor expansions above into \eqref{U} and \eqref{extra} in order to conclude that $s_1$ is constant. The computations are elementary but tedious. In order to clarify them, we start by sketching the main steps without going into details. The complete proof follows below.

\vspace{.2em}
\textit{Sketch of the main steps:}
Plugging \eqref{taylorn} into \eqref{extra} leads to an equation of the form
\begin{equation}\label{taylorextra}
0 = (r-1)^3 A_3 + (r-1)^4 A_4 + O((r-1)^5),
\end{equation}
where
\begin{align*}
A_3 & = A_3(m_1,\partial_\theta m_1,\partial_\varphi m_1),\\
A_4 & = A_4(m_1,m_2,\partial_\theta m_1,\partial_\varphi m_1,\partial_\theta m_2 , \partial_\varphi m_2).
\end{align*}
At this point, a first simplification occurs, since $A_4$ is actually of the form
\begin{equation*}
A_4 = -2 A_3 + \widetilde A_4(m_2 ,\partial_\theta m_2,\partial_\varphi m_2),
\end{equation*}
so that from \eqref{taylorextra} we deduce
\begin{equation}\label{tildeA4}
\widetilde A_4(m_2,\partial_\theta m_2,\partial_\varphi m_2) = 0.
\end{equation}

Next we make use of \eqref{U}. Plugging \eqref{taylorn} and \eqref{taylors} into \eqref{U}, we obtain equations of the form
\begin{equation}\label{taylorU}
\left\lbrace \begin{aligned}
0 & = \alpha_0 + O(r-1) \\
0 & = (r-1)v_1 + (r-1)^2 v_2 + O((r-1)^3,
\end{aligned}\right.
\end{equation}
where
\begin{align*}
\alpha_0 &= \alpha_0 (s_0,s_1,s_2),\\
v_1 & = v_1(s_0,m_1,\partial_\theta s_1,\partial_\varphi s_1 ) \\
v_2 & = v_2(s_0,s_1,m_1,m_2,\partial_\theta s_1,\partial_\varphi s_1,\partial_\theta s_2,\partial_\varphi s_2).
\end{align*}
The first equation in \eqref{taylorU} implies that $\alpha_0=0$. Solving $\alpha_0=0$, we obtain an expression of $s_2$ in terms of $s_1$ and $s_0$, which we plug into $v_2$. Here a new simplification arises: it holds
\begin{equation*}
v_2 = \tilde v_2 (s_0,s_1,m_1,m_2) - 3 v_1.
\end{equation*}
Thus \eqref{taylorU} implies that $v_1=\tilde v_2 = 0$. Solving $\tilde v_2 =0$ we find an expression
\begin{equation*}
m_2 = m_2 (s_0,s_1,m_1),
\end{equation*}
which we plug into \eqref{tildeA4} to obtain an equation of the form
\begin{equation*}
 A_4^\ast (s_0,s_1,m_1,\partial_\theta s_1,\partial_\varphi s_1,\partial_\theta m_1,\partial_\varphi m_1)=0.
\end{equation*}
Using the equation $A_3=0$ from \eqref{taylorextra}, we are able to simplify the last expression of $A_4^\ast$ into one which does not involve derivatives of $m_1$:
\begin{equation}\label{tildeA4again}
\widehat A_4(s_0,s_1,m_1,\partial_\theta s_1,\partial_\varphi s_1) = 0.
\end{equation}
Eventually we use the equation $v_1=0$ to express $m_1$ in terms of $s_0$, $\partial_\theta s_1$ and $\partial_\varphi s_1$. Plugging that expression of $m_1$ into \eqref{tildeA4again} leads us to a system of the form
\begin{equation*}
A_4^\sharp (s_0,\partial_\theta s_1,\partial_\varphi s_1)=0.
\end{equation*}
The above equation forces $\partial_\theta s_1=\partial_\varphi s_1 = 0$ and thus allows to conclude.

\vspace{.2em}
\textit{Complete proof:}
It holds
\begin{align*}
\partial_r n & = 3(r-1)^2 m_1 + 4(r-1)^3 m_2 + O((r-1)^4), \\
\partial_r^2 n& =  6(r-1)m_1 + 12(r-1)^2m_2+O((r-1)^3),\\
\partial_\theta n & = \mathbf{e_\theta} + (r-1)^3 \partial_\theta m_1 + (r-1)^4\partial_\theta m_2 + O((r-1)^5),\\
\partial_\varphi n &= \sin\theta \mathbf{e_\varphi} + (r-1)^3 \partial_\varphi m_1 + (r-1)^4\partial_\varphi m_2 + O((r-1)^5),\\
\Delta_{\mathbb S^2} n & = -2 \mathbf{e_r} + O((r-1)^3),
\end{align*}
and thus
\begin{align*}
\Delta n & = \partial_r^2 n + \frac{2}{r}\partial_r n + \frac{1}{r^2}\Delta_{\mathbb S^2} n \\
& = \partial_r^2 n + 2 (1+O((r-1))\partial_r n \\
& \quad  + \left(1-2(r-1)+3(r-1)^2+O((r-1)^3\right)\left(-2\mathbf{e_r}+O((r-1)^3)\right) \\
& = 6(r-1)m_1 + 12(r-1)^2 m_2 + 6 (r-1)^2 m_1 \\
& \quad -2 \mathbf{e_r} + 4(r-1) \mathbf{e_r} -6 (r-1)^2\mathbf{e_r} + O((r-1)^3) \\
&=-2\mathbf{e_r} + (r-1) \left[ 6 m_1+4\mathbf{e_r} \right] + (r-1)^2 \left[ 12 m_2 + 6m_1 -6 \mathbf{e_r}\right] + O((r-1)^3).
\end{align*}
Hence we compute
\begin{align*}
s\Delta n & = -2 s_0 \mathbf{e_r} + (r-1)\left[ 6 s_0 m_1 + 4s_0\mathbf{e_r} -2 s_1 \mathbf{e_r}\right]\\
&\quad +(r-1)^2 \left[ 12 s_0 m_2 + 6s_0 m_1 -6s_0\mathbf{e_r} +6 s_1 m_1 + 4 s_1 \mathbf{e_r} -2 s_2 \mathbf{e_r}\right] \\
& \quad + O((r-1)^3).
\end{align*}
Next we want to compute
\begin{equation*}
(\nabla s \cdot\nabla)n = \partial_r s\partial_r n + \frac{1}{r^2}\partial_\theta s\partial_\theta n + \frac{1}{r^2\sin^2\theta}\partial_\varphi s\partial_\varphi n.
\end{equation*}
We calculate each term:
\begin{align*}
\partial_r s\partial_r n& = \left(s_1 + 2(r-1)s_2+O((r-1)^2)\right)\left(3(r-1)^2m_1+O((r-1)^3)\right) \\
& = 3s_1(r-1)^2 m_1 + O((r-1)^3),\\
\frac{1}{r^2}\partial_\theta s\partial_\theta n & = 
\frac{1}{r^2}\left( (r-1)\partial_\theta s_1 + (r-1)^2\partial_\theta s_2 + O((r-1)^3)\right) \\
& \quad\times \left(\mathbf{e_\theta}+O((r-1)^3)\right) \\
& = \left(1-2(r-1)+O(r-1)\right)\\
& \quad \times \left( (r-1)\partial_\theta s_1 \mathbf{e_\theta} + (r-1)^2\partial_\theta s_2 \mathbf{e_\theta} + O((r-1)^3)\right)\\
& = (r-1) \partial_\theta s_1 \mathbf{e_\theta} \\
& \quad + (r-1)^2\left[ \partial_\theta s_2 \mathbf{e_\theta} -2 \partial_\theta s_1 \mathbf{e_\theta} \right ] + O((r-1)^3),\\
\frac{1}{r^2\sin^2\theta}\partial_\varphi s \partial_\varphi n & =
(r-1) \frac{\partial_\varphi s_1}{\sin\theta} \mathbf{e_\varphi} \\
& \quad + (r-1)^2\left[ \frac{\partial_\varphi s_2}{\sin\theta} \mathbf{e_\varphi} -2 \frac{\partial_\varphi s_1}{\sin\theta} \mathbf{e_\varphi} \right ] + O((r-1)^3).
\end{align*}

Thus it holds:
\begin{align*}
s\Delta n + 2(\nabla s\cdot\nabla)n & =
-2 s_0 \mathbf{e_r} \\
&\quad  + (r-1)\Big[ 6 s_0 m_1 +4s_0 \mathbf{e_r} -2s_1\mathbf{e_r} + 2\partial_\theta s_1 \mathbf{e_\theta} + 2\frac{\partial_\varphi s_1}{\sin\theta}\mathbf{e_\varphi}\Big] \\
& \quad + (r-1)^2\Big[
12s_0m_2 +6s_0m_1-6s_0\mathbf{e_r}+12s_1m_1+4s_1\mathbf{e_r} \\
& \qquad\qquad -2s_2\mathbf{e_r}+2\partial_\theta s_2 \mathbf{e_\theta} -4\partial_\theta s_1 \mathbf{e_\theta} + 2\frac{\partial_\varphi s_2}{\sin\theta}\mathbf{e_\varphi} -4\frac{\partial_\varphi s_1}{\sin\theta}\mathbf{e_\varphi}
\Big] \\
& \quad + O((r-1)^3).
\end{align*}

Our next step is to compute the symmetric matrix
\begin{equation*}
M = \partial_r n \otimes \partial_r n + \frac{1}{r^2}\partial_\theta n\otimes\partial_\theta n + \frac{1}{r^2\sin^2\theta}\partial_\varphi n \otimes\partial_\varphi n.
\end{equation*}

We compute each term:
\begin{align*}
\partial_r n\otimes\partial_r n & = 9 (r-1)^4 m_1\otimes m_1 + O((r-1)^5),\\
\frac{1}{r^2}\partial_\theta n\otimes\partial_\theta n & =
\left( 1 -2(r-1)+3(r-1)^2-4(r-1)^3+5(r-1)^4 + O((r-1)^5)\right) \\
& \quad \times \left( \mathbf{e_\theta}\otimes \mathbf{e_\theta} + 2(r-1)^3 \partial_\theta m_1 \odot \mathbf{e_\theta} + 2(r-1)^4 \partial_\theta m_2\odot \mathbf{e_\theta} + O((r-1)^5)\right)\\
& = \mathbf{e_\theta} \otimes \mathbf{e_\theta} -2 (r-1) \mathbf{e_\theta}\otimes  \mathbf{e_\theta} + 3 (r-1)^2 \mathbf{e_\theta}\otimes \mathbf{e_\theta} \\
& \quad + (r-1)^3 \left[ -4 \mathbf{e_\theta}\otimes \mathbf{e_\theta} + 2\partial_\theta m_1 \odot \mathbf{e_\theta} \right]\\
&\quad + (r-1)^4 \left[ 5\mathbf{e_\theta}\otimes \mathbf{e_\theta} -4 \partial_\theta m_1 \odot \mathbf{e_\theta} + 2 \partial_\theta m_2 \odot \mathbf{e_\theta} \right] \\
& \quad + O((r-1)^5),\\
\frac{1}{r^2\sin^2\theta}\partial_\varphi n \otimes \partial_\varphi n & = 
\mathbf{e_\varphi} \otimes \mathbf{e_\varphi} -2 (r-1) \mathbf{e_\varphi}\otimes  \mathbf{e_\varphi} + 3 (r-1)^2 \mathbf{e_\varphi}\otimes \mathbf{e_\varphi} \\
& \quad + (r-1)^3 \left[ -4 \mathbf{e_\varphi}\otimes \mathbf{e_\varphi} + \frac{2}{\sin\theta}\partial_\varphi m_1 \odot \mathbf{e_\varphi} \right]\\
&\quad + (r-1)^4 \left[ 5\mathbf{e_\varphi}\otimes \mathbf{e_\varphi} -\frac{4}{\sin\theta} \partial_\varphi m_1 \odot \mathbf{e_\varphi} + \frac{2}{\sin\theta} \partial_\varphi m_2 \odot \mathbf{e_\varphi} \right] \\
& \quad + O((r-1)^5).
\end{align*}

Hence we have
\begin{equation*}
M=M_0 + (r-1)M_1 +\cdots + (r-1)^4 M_4 + O((r-1)^5),
\end{equation*}
where
\begin{align*}
M_0 & = \mathbf{e_\theta}\otimes \mathbf{e_\theta} + \mathbf{e_\varphi}\otimes \mathbf{e_\varphi} = I-\mathbf{e_r} \otimes \mathbf{e_r},\\
M_1 & = -2 (I-\mathbf{e_r}\otimes \mathbf{e_r}), \\
M_2 & =  3 (I-\mathbf{e_r}\otimes \mathbf{e_r}), \\
M_3 & =  -4 (I-\mathbf{e_r}\otimes \mathbf{e_r}) + 2\partial_\theta m_1\odot \mathbf{e_\theta} + \frac{2}{\sin\theta}\partial_\varphi m_1\odot \mathbf{e_\varphi},\\
M_4 & =9 m_1 \otimes m_1 + 5 (I-\mathbf{e_r}\otimes \mathbf{e_r}) - 4\partial_\theta m_1\odot \mathbf{e_\theta} - \frac{4}{\sin\theta}\partial_\varphi m_1\odot \mathbf{e_\varphi} \\
& \quad + 2\partial_\theta m_2\odot \mathbf{e_\theta} + \frac{2}{\sin\theta}\partial_\varphi m_2\odot \mathbf{e_\varphi}.\\
\end{align*}

Using the fact that $|\nabla n|^2 =\mathrm{tr}\: M$, we obtain
in particular
\begin{align*}
|\nabla n|^2 &= 2 - 4(r-1) + 6 (r-1)^2 + (r-1)^3\Big[ -8 +
2 \partial_\theta m_1\cdot \mathbf{e_\theta}
+\frac{2}{\sin\theta}\partial_\varphi m_1\cdot \mathbf{e_\varphi} \Big]\\
&\quad +(r-1)^4\Big[9|m_1|^2+10 - 4\partial_\theta m_1\cdot \mathbf{e_\theta} -
\frac{4}{\sin\theta}\partial_\varphi m_1\cdot \mathbf{e_\varphi}+
2\partial_\theta m_2\cdot \mathbf{e_\theta} +
\frac{2}{\sin\theta}\partial_\varphi m_2\cdot \mathbf{e_\varphi} \Big]\\
&\quad +O((r-1)^5),
\end{align*}
and
\begin{align*}
|\nabla n|^2n&=2\mathbf{e_r}-4(r-1)\mathbf{e_r}+6(r-1)^2\mathbf{e_r}\\
&\quad +(r-1)^3\Big[( -8 +
2 \partial_\theta m_1\cdot \mathbf{e_\theta}
+\frac{2}{\sin\theta}\partial_\varphi m_1\cdot \mathbf{e_\varphi})\mathbf{e_r} +2m_1
\Big]\\
&\quad +(r-1)^4\Big[\big\{9|m_1|^2+10 - 4\partial_\theta m_1\cdot \mathbf{e_\theta} -
\frac{4}{\sin\theta}\partial_\varphi m_1\cdot \mathbf{e_\varphi}+
2\partial_\theta m_2\cdot \mathbf{e_\theta}\\
&\qquad\qquad\qquad +
\frac{2}{\sin\theta}\partial_\varphi m_2\cdot \mathbf{e_\varphi}\big\}\mathbf{e_r} -4m_1
+2m_2 \Big]\\
&\quad + O((r-1)^5),
\end{align*}
\begin{align*}
s|\nabla n|^2n & = 2s_0 \mathbf{e_r} + (r-1)\big[ 2s_1-4s_0\big]\mathbf{e_r}
+(r-1)^2\big[6s_0-4s_1+2s_2\big]\mathbf{e_r}\\
&\quad + O((r-1)^3),
\end{align*}
\begin{align*}
|\nabla n|^2 n\otimes n & = 2\mathbf{e_r}\otimes \mathbf{e_r}-4(r-1)\mathbf{e_r}\otimes \mathbf{e_r}+6(r-1)^2\mathbf{e_r}\otimes \mathbf{e_r}\\
&\quad +(r-1)^3\Big[( -8 +
2 \partial_\theta m_1\cdot \mathbf{e_\theta}
+\frac{2}{\sin\theta}\partial_\varphi m_1\cdot \mathbf{e_\varphi})\mathbf{e_r}\otimes
\mathbf{e_r} +4m_1\odot \mathbf{e_r}
\Big]\\
&\quad +(r-1)^4\Big[\big\{9|m_1|^2+10 - 4\partial_\theta m_1\cdot \mathbf{e_\theta} -
\frac{4}{\sin\theta}\partial_\varphi m_1\cdot \mathbf{e_\varphi}+
2\partial_\theta m_2\cdot \mathbf{e_\theta}\\
&\qquad\qquad\qquad +
\frac{2}{\sin\theta}\partial_\varphi m_2\cdot
\mathbf{e_\varphi}\big\}\mathbf{e_r}\otimes \mathbf{e_r} -8m_1\odot \mathbf{e_r}
+4m_2\odot \mathbf{e_r} \Big]\\
&\quad + O((r-1)^5),
\end{align*}
\begin{align*}
|\nabla n|^2(I-n\otimes n) & = 2(I-\mathbf{e_r}\otimes \mathbf{e_r})-4(r-1)(I-\mathbf{e_r}\otimes \mathbf{e_r})+6(r-1)^2(I-\mathbf{e_r}\otimes \mathbf{e_r})\\
&\quad +(r-1)^3\Big[( -8 +
2 \partial_\theta m_1\cdot \mathbf{e_\theta}
+\frac{2}{\sin\theta}\partial_\varphi m_1\cdot \mathbf{e_\varphi})(I-\mathbf{e_r}\otimes \mathbf{e_r}) -4m_1\odot \mathbf{e_r}
\Big]\\
&\quad +(r-1)^4\Big[\big\{9|m_1|^2+10 - 4\partial_\theta m_1\cdot \mathbf{e_\theta} -
\frac{4}{\sin\theta}\partial_\varphi m_1\cdot \mathbf{e_\varphi}+
2\partial_\theta m_2\cdot \mathbf{e_\theta}\\
&\qquad\qquad\qquad +
\frac{2}{\sin\theta}\partial_\varphi m_2\cdot
\mathbf{e_\varphi}\big\}(I-\mathbf{e_r}\otimes \mathbf{e_r}) +8m_1\odot \mathbf{e_r}
-4m_2\odot \mathbf{e_r} \Big]\\
&\quad + O((r-1)^5).
\end{align*}
Eventually, we have:
\begin{align*}
s\Delta n+2(\nabla s\cdot\nabla)n+s|\nabla n|^2n & =
 (r-1)\Big[6 s_0 m_1 + 2\partial_\theta s_1 \mathbf{e_\theta}+2\frac{\partial_\varphi s_1}{\sin\theta}\mathbf{e_\varphi} 
\Big]\\
&\quad + (r-1)^2\Big[12s_0m_2 +6s_0m_1+12s_1m_1
+2\partial_\theta s_2 \mathbf{e_\theta}\\
&\qquad\qquad\qquad -4\partial_\theta s_1 \mathbf{e_\theta} +
2\frac{\partial_\varphi s_2}{\sin\theta}\mathbf{e_\varphi}
-4\frac{\partial_\varphi s_1}{\sin\theta}\mathbf{e_\varphi} \Big]\\
&\quad +O((r-1)^3),
\end{align*}
and
\begin{equation*}
2M-|\nabla n|^2(I-n\otimes n) = (r-1)^3 A_3 + (r-1)^4A_4 + O((r-1)^5),
\end{equation*}
where
\begin{align*}
A_3 & = 4\partial_\theta m_1\odot \mathbf{e_\theta} +
\frac{4}{\sin\theta}\partial_\varphi m_1 \odot \mathbf{e_\varphi} -
\Big[ 2\partial_\theta m_1\cdot \mathbf{e_\theta} +
\frac{2}{\sin\theta}\partial_\varphi m_1 \cdot
\mathbf{e_\varphi}\Big](I-\mathbf{e_r}\otimes \mathbf{e_r})\\
&\quad +4m_1 \odot \mathbf{e_r},\\
A_4 & =18 m_1\otimes m_1 - 8\partial_\theta m_1 \odot \mathbf{e_\theta} -
\frac{8}{\sin\theta}\partial_\varphi m_1 \odot \mathbf{e_\varphi} +
4\partial_\theta m_2 \odot \mathbf{e_\theta}
+\frac{4}{\sin\theta}\partial_\varphi m_2 \odot \mathbf{e_\varphi} \\
&\quad -\Big[9 |m_1|^2 - 4\partial_\theta m_1 \cdot \mathbf{e_\theta} -
\frac{4}{\sin\theta}\partial_\varphi m_1 \cdot \mathbf{e_\varphi} +
2\partial_\theta m_2 \cdot \mathbf{e_\theta}+
\frac{2}{\sin\theta}\partial_\varphi m_2 \odot
\mathbf{e_\varphi}\Big](I-\mathbf{e_r}\otimes \mathbf{e_r})\\
& \quad -8 m_1 \odot \mathbf{e_r} + 4 m_2 \odot \mathbf{e_r} \\
& = -2 A_3 + 18 m_1\otimes m_1 + 4\partial_\theta m_2 \odot \mathbf{e_\theta}
+\frac{4}{\sin\theta}\partial_\varphi m_2 \odot \mathbf{e_\varphi} + 4 m_2 \odot \mathbf{e_r}  \\
&\quad - \Big[9 |m_1|^2 +
2\partial_\theta m_2 \cdot \mathbf{e_\theta}+
\frac{2}{\sin\theta}\partial_\varphi m_2 \odot
\mathbf{e_\varphi}\Big](I-\mathbf{e_r}\otimes \mathbf{e_r}).
\end{align*}

Moreover, denoting by
\begin{equation*}
\psi(s):=\frac 1L \left( 2s\partial_1\varphi(2s^2/3,2s^3/9)+s^2\partial_2\varphi (2s^2/3,2s^3/9) \right)
\end{equation*}
the nonlinear term of order 0 arising in the first equation of \eqref{U}, we have
\begin{equation*}
\Delta s -3s|\nabla n|^2-\psi(s) = 2s_2 +2 s_1 - 6s_0 + \psi(s_0)+O(r-1).
\end{equation*}

We conclude that the following equalities hold:
\begin{gather}
s_2=-s_1+3s_0+\frac 12 \psi(s_0),\label{*1}\\
6 s_0 m_1 + 2\partial_\theta s_1 \mathbf{e_\theta} +2\frac{\partial_\varphi s_1}{\sin\theta}\mathbf{e_\varphi}=0, \label{*2a}\\
12s_0m_2 +6s_0m_1+12 s_1m_1
+2\partial_\theta s_2 \mathbf{e_\theta}
-4\partial_\theta s_1 \mathbf{e_\theta} +
2\frac{\partial_\varphi s_2}{\sin\theta}\mathbf{e_\varphi}
-4\frac{\partial_\varphi s_1}{\sin\theta}\mathbf{e_\varphi} = 0, \label{*2b}\\
\begin{split}
4\partial_\theta m_1\odot \mathbf{e_\theta} & +
\frac{4}{\sin\theta}\partial_\varphi m_1 \odot \mathbf{e_\varphi}  +4m_1 \odot \mathbf{e_r} = \\
& \Big[ 2\partial_\theta m_1\cdot \mathbf{e_\theta} +
\frac{2}{\sin\theta}\partial_\varphi m_1 \cdot
\mathbf{e_\varphi}\Big](I-\mathbf{e_r}\otimes \mathbf{e_r}),
\end{split}\label{Ca} \\
\begin{split}
& 18 m_1\otimes m_1  +
4\partial_\theta m_2 \odot \mathbf{e_\theta}
+\frac{4}{\sin\theta}\partial_\varphi m_2 \odot \mathbf{e_\varphi} 
 + 4 m_2 \odot \mathbf{e_r}\\
& \quad = \Big[9 |m_1|^2  +
2\partial_\theta m_2 \cdot \mathbf{e_\theta}+
\frac{2}{\sin\theta}\partial_\varphi m_2 \odot
\mathbf{e_\varphi}\Big](I-\mathbf{e_r}\otimes \mathbf{e_r}).
\end{split}\label{Cb}
\end{gather}
Since $\psi(s_0)$ is a constant, \eqref{*1} implies that
\begin{equation*}
\partial_\theta s_2 = -\partial_\theta s_1,\quad\partial_\varphi s_2 = - \partial_\varphi s_1,
\end{equation*}
so that  \eqref{*2b} becomes
\begin{equation*}
12s_0m_2 +6s_0m_1+12 s_1m_1
=6\partial_\theta s_1 \mathbf{e_\theta} + 6\frac{\partial_\varphi s_1}{\sin\theta}\mathbf{e_\varphi}
\end{equation*}
that is, using \eqref{*2a},
\begin{equation*}
12s_0m_2 + 24 s_0 m_1 + 12 s_1 m_1 = 0
\end{equation*}
from which we deduce an expression of $m_2$ in terms of $s_0$, $s_1$ and $m_1$:
\begin{equation}\label{*2b'}
m_2 = -\frac{2s_0+s_1}{s_0}m_1.
\end{equation}
Thus we compute, using also \eqref{Ca},
\begin{align*}
4\partial_\theta m_2 \odot \mathbf{e_\theta} + \frac{4}{\sin\theta}\partial_\varphi m_2 \odot \mathbf{e_\varphi} & =-\frac{2s_0+s_1}{s_0}\left( 4\partial_\theta m_1 \odot \mathbf{e_\theta} + \frac{4}{\sin\theta}\partial_\varphi m_1 \odot \mathbf{e_\varphi} \right) \\
& \quad - \frac{1}{s_0} \left(\partial_\theta s_1 m_1 \odot \mathbf{e_\theta}+\frac{1}{\sin\theta}\partial_\varphi s_1 m_1\odot \mathbf{e_\varphi}\right)\\
& = -\frac{2s_0+s_1}{s_0}\big[ 2\partial_\theta m_1\cdot \mathbf{e_\theta} +
\frac{2}{\sin\theta}\partial_\varphi m_1 \cdot
\mathbf{e_\varphi}\big](I-\mathbf{e_r}\otimes \mathbf{e_r})\\
& \quad +4 \frac{2s_0+s_1}{s_0} m_1\odot \mathbf{e_r} \\
& \quad - \frac{1}{s_0} \left(\partial_\theta s_1 m_1 \odot \mathbf{e_\theta}+\frac{1}{\sin\theta}\partial_\varphi s_1 m_1\odot \mathbf{e_\varphi}\right)\\
& = \big[ 2\partial_\theta m_2 \cdot \mathbf{e_\theta} + \frac{2}{\sin\theta}\partial_\varphi m_2 \cdot \mathbf{e_\varphi} \big] (I-\mathbf{e_r}\otimes \mathbf{e_r})\\
&\quad +\frac{1}{2s_0}\big[ \partial_\theta s_1 m_1 \cdot \mathbf{e_\theta}+\frac{1}{\sin\theta}\partial_\varphi s_1 m_1\cdot \mathbf{e_\varphi}\big](I-\mathbf{e_r} \otimes \mathbf{e_r})\\
& \quad -4m_2\odot \mathbf{e_r} \\
& \quad - \frac{1}{s_0} \left(\partial_\theta s_1 m_1 \odot \mathbf{e_\theta}+\frac{1}{\sin\theta}\partial_\varphi s_1 m_1\odot \mathbf{e_\varphi}\right).
\end{align*}
We plug this last computation into \eqref{Cb}, which gives:
\begin{equation}\label{Cab'}
\begin{split}
& 18 m_1 \otimes m_1 - 9 |m_1|^2 (I-\mathbf{e_r}\otimes \mathbf{e_r})= \\
& \qquad \frac 1 {s_0} \left(\partial_\theta s_1 m_1 \odot \mathbf{e_\theta}+\frac{1}{\sin\theta}\partial_\varphi s_1 m_1\odot \mathbf{e_\varphi}\right) \\
& \qquad  -\frac 1 {2s_0} \big[ \partial_\theta s_1 m_1 \cdot \mathbf{e_\theta}+\frac{1}{\sin\theta}\partial_\varphi s_1 m_1\cdot \mathbf{e_\varphi}\big](I-\mathbf{e_r} \otimes \mathbf{e_r}).
\end{split}
\end{equation}
The identity \eqref{Cab'} is an equality of symmetric (traceless) matrices, so it amounts to  5 scalar equalities. Actually only two of them are interesting (see Remark~\ref{remextra}). In the sequel we are going to make use of \eqref{Cab'} applied -- as an equality of bilinear forms -- to $(\mathbf{e_\theta},\mathbf{e_\theta})$ and $(\mathbf{e_\theta},\mathbf{e_\varphi})$, which gives the two following equations:
\begin{equation}
\label{Cab}
\begin{split}
18 (m_1\cdot \mathbf{e_\theta} )^2 - 9 |m_1|^2 &  = \frac 1 {2s_0} \partial_\theta s_1 m_1\cdot \mathbf{e_\theta} - \frac 1 {2 s_0\sin\theta} \partial_\varphi s_1 m_1 \cdot \mathbf{e_\varphi} \\
18 (m_1\cdot \mathbf{e_\theta})(m_1\cdot \mathbf{e_\varphi}) & = \frac 1 {2 s_0} \partial_\theta s_1 m_1\cdot \mathbf{e_\varphi} +\frac 1 {2s_0\sin\theta}\partial_\varphi s_1 m_1 \cdot \mathbf{e_\theta}
\end{split}
\end{equation}
Eventually we make use of \eqref{*2a} to transform \eqref{Cab'} into equations involving only the derivatives of $s_1$. 

Equation \eqref{*2a} may indeed be rewritten as
\begin{equation*}
m_1 = -\frac{1}{3s_0} \partial_\theta s_1 \mathbf{e_\theta} - \frac{1}{3s_0\sin\theta} \partial_\varphi s_1 \mathbf{e_\varphi}.
\end{equation*}
Thus we have the following identities:
\begin{gather*}
m_1\cdot \mathbf{e_\theta}  = -\frac{1}{3s_0}\partial_\theta s_1, \quad
m_1\cdot \mathbf{e_\varphi}  = -\frac{1}{3s_0\sin\theta}\partial_\varphi s_1, \\
|m_1|^2  = \frac{1}{9 s_0^2}\left( (\partial_\theta s_1)^2 + \frac{(\partial_\varphi s_1)^2}{\sin^2\theta}\right)
\end{gather*}
which we plug into \eqref{Cab'} to obtain:
\begin{gather*}
\frac{2}{s_0^2}(\partial_\theta s_1)^2-\frac{1}{s_0^2}\left( (\partial_\theta s_1)^2 + \frac{(\partial_\varphi s_1)^2}{\sin^2\theta}\right) 
= -\frac 1 {6 s_0^2} (\partial_\theta s_1)^2 + \frac 1 {6 s_0^2\sin^2\theta} (\partial_\varphi s_1)^2\\
\frac{2}{s_0^2\sin\theta}(\partial_\theta s_1)(\partial_\varphi s_1) =
-\frac {1} {3s_0^2\sin \theta}(\partial_\theta s_1)(\partial_\varphi s_1)
\end{gather*}
i.e.
\begin{gather*}
 (\partial_\theta s_1)^2 - \frac{1}{\sin^2\theta}(\partial_\varphi s_1)^2 = 0 \\
\frac{1}{\sin\theta}(\partial_\theta s_1)(\partial_\varphi s_1) = 0
\end{gather*}
 Clearly, the last equations imply that
\begin{equation*}
\partial_\theta s_1 = \partial_\varphi s_1 = 0,
\end{equation*}
which proves \textit{Step 3}.
\end{proof}

\section{Conclusions and perspectives}

\subsection{Conclusions}

We have studied nematic equilibrium configurations under the constraint of uniaxial symmetry. The results we have obtained show that the constraint of uniaxial symmetry is very restrictive and should in general not be satisfied by equilibrium configurations, except in the presence of other strong symmetries. 

We have shown that, for a nematic equilibrium configuration presenting translational invariance in one direction, there are only two options: either it does not have any regions with uniaxial symmetry, or it has uniform director field. In particular, when the boundary conditions prevent the director field from being uniform, as it is the case in hybrid cells or in capillaries with radial anchoring, then at equilibrium uniaxial order is destroyed spontaneously within the whole system. In other words, for translationally invariant configurations, biaxial escape has to occur.

Biaxiality had in fact been predicted in such geometries \cite{sonnetkilianhess95, palffymuhoraygartlandkelly94, bisigartlandrossovirga03}, but it was supposed to stay confined to small regions, and to occur only in some parameter range. Here we have provided a rigorous proof that biaxiality must occur everywhere, and for any values of the parameter: the configurations interpreted as uniaxial just correspond to a small degree of biaxiality. Our proof does not rely on free energy minimization, but only on the equilibrium equations -- in particular it affects all metastable configurations.  It is also remarkable that our results do not depend on the form of the bulk energy density, whereas all the previously cited workers used a four-terms approximation.

For general three-dimensional configurations we have not obtained a complete description of uniaxial equilibrium configurations, but we have studied the model case of the hedgehog defect, and obtained a strong symmetry result: a uniaxial equilibrium must be spherically symmetric. We believe in fact that, in general, the only non trivial uniaxial solutions of the equilibrium equation are spherically symmetric.

\subsection{Perspectives}

Many interesting problems concerning uniaxial equilibrium and biaxial escape remain open. We mention here three directions of further research. 

The first one is the complete description of three-dimensional uniaxial solutions of \eqref{EL}. Techniques similar to the proof of Theorem~\ref{thmrad} should allow to prove that, in a smooth bounded domain with normal anchoring, uniaxial solutions exist only if the domain has spherical symmetry. Such a result would constitute a first step towards the conjectured fact that the only non trivial uniaxial solution of \eqref{EL} -- whatever the form of the domain and the boundary conditions -- are spherically symmetric. For more general boundary conditions however, other techniques would likely be needed.

Another open problem is to consider more general (and more physically relevant) elastic terms (see Remark~\ref{remframeinv}). The equation \eqref{extra} corresponding to equilibrium with respect to symmetry-breaking perturbations is more complicated in that case (in particular it is of second order).

A third problem, which is of even greater physical relevance, is to investigate ``approximately uniaxial'' equilibrium configurations. Hopefully, equation \eqref{extra} could play an interesting role in such a study.

\bibliographystyle{plain}
\bibliography{uniax}

\end{document}